\documentclass[12pt]{article}
\usepackage{amsmath,amssymb,amsthm}
\usepackage{geometry}
\usepackage{graphicx}
\usepackage[T1]{fontenc}
\usepackage[numbers,sort,compress]{natbib}
\newtheorem{theorem}{Theorem}[section]
\newtheorem{lemma}[theorem]{Lemma}

\newtheorem{proposition}[theorem]{Proposition}

\newtheorem{example}[theorem]{Example}
\title{On the generalized cospectral mates of chain graphs of rank 4}
\author
{
	Wei Wang\thanks{Corresponding author\\
		E-mail addresses:  wangwei.math@gmail.com (W. Wang); 3132237735@qq.com (S. Wu).}\quad\quad
	Siqi Wu\\
	\footnotesize School of Mathematics, Physics and Finance, Anhui Polytechnic University,
	Wuhu 241000, China
}
\date{}

\begin{document}
	\maketitle
	
	\begin{abstract}
	A graph is called a chain graph if it is bipartite and the
	neighborhoods of the vertices in each color class are linearly ordered
	by inclusion. We study the generalized spectral determination of chain
	graphs of adjacency rank $4$. We first show that generalized cospectral
	chain graphs have the same number of isolated vertices, reducing the
	problem to the connected case. For connected rank-$4$ chain graphs, we
	derive four explicit equations in the cell parameters that are necessary
	and sufficient for generalized cospectrality. These equations reduce to
	a quadratic equation in a single parameter. As a consequence, every
	connected rank-$4$ chain graph has at most one  generalized
	cospectral mate, and the possible mate can be reconstructed explicitly.
We further obtain a simple sufficient condition for a rank-$4$ chain graph to be determined by its generalized spectrum (DGS) and, in particular, prove that every double star graph is DGS. Finally, we construct an explicit infinite family of non-double-star rank-$4$ chain graphs that are DGS but not determined by their adjacency spectra, with both the graphs and their complements being connected.
\end{abstract}
	
		\noindent\textbf{Keywords:} Graph spectra; 
		Cospectral graphs; 
		Determined by spectrum (DS); 
		Determined by generalized spectrum
		(DGS); Chain graphs.\\
	
	\noindent\textbf{Mathematics Subject Classification:} 05C50; 05C75; 05C60
	
	\section{Introduction}
	Let $G$ be a simple graph with adjacency matrix $A(G)$. The characteristic polynomial of $G$ is
	\[
	\chi(x,G)=\det(xI-A(G)),
	\]
	and its roots constitute the spectrum of $G$. Two graphs $G$ and $H$ are said to be \emph{cospectral} if they have the same spectrum. Furthermore, $G$ and $H$ are called \emph{generalized cospectral} if they are cospectral and their complements $\overline{G}$ and $\overline{H}$ are also cospectral. We say $H$ is a generalized cospectral mate of $G$ if $G$ and $H$ are \emph{generalized cospectral} but are not isomorphic.
	
	A graph $G$ is said to be \emph{determined by its spectrum} (DS) if every graph cospectral with $G$ is isomorphic to $G$. Similarly, $G$ is said to be \emph{determined by its generalized spectrum} (DGS) if every graph generalized cospectral with $G$ is isomorphic to $G$. A fundamental and difficult problem in spectral graph theory is to characterize graphs that are DS or DGS. For background and recent developments, we refer the reader to \cite{vanDam2003,vanDam2009,Sason2025,Koval2024,Wang2017,Wang2025,Guo2025}.
	
	Although every DS graph is necessarily DGS, the converse does not hold in general. This distinction leads to the natural question posed by Lin et al.~\cite{Lin2026}: \emph{Which graphs are DGS but not DS?} As noted in Remark~3 of \cite{Lin2026}, infinite families of such graphs can be obtained by relatively simple constructions when disconnectedness is allowed for either the graph or its complement. For example, certain stars and disjoint unions of complete bipartite graphs with isolated vertices yield DGS but non-DS graphs because their complements are DS. It is therefore natural to seek nontrivial infinite families for which both the graphs and their complements are connected. Here $D(s,t)$, where $s,t\ge1$, denotes the double star obtained by joining the centers of $K_{1,s}$ and $K_{1,t}$. Lin et al.~\cite{Lin2026} made progress in this direction by investigating the family
	\[
	\mathcal{D}_1:=\{D(k+3,2k+3)\colon\,k\geq 0\}
	\]
	(illustrated in Fig.~\ref{ds}), and proving that every graph in this family is DGS but not DS. Both the members of $\mathcal{D}_1$ and their complements are connected.

	\begin{figure}[htbp]
		\centering
		\includegraphics[width=0.5\textwidth]{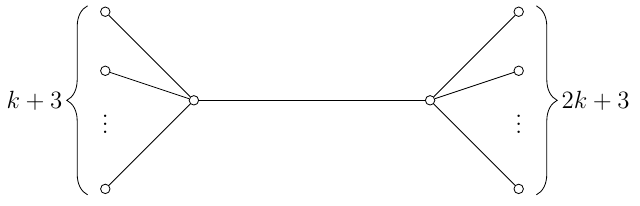}
		\caption{The double star graph $D(k+3,2k+3)$.}
		\label{ds}
	\end{figure}
	More recently, Zhu \cite{zhu2026} considered the family
	\[
	\mathcal{D}_2:=\{D(2k,k+1)\colon\, k\ge2\},
	\]
	and proved that $D(2k,k+1)$ is DGS but not DS whenever neither $k$ nor
	$2k+2$ is a perfect square. The proofs of Lin et al.~\cite{Lin2026}
	and Zhu \cite{zhu2026} are quite different. We prove the following
	unifying result.
	\begin{theorem}\label{main}
		Every double star graph is DGS.
	\end{theorem}
Theorem~\ref{main} shows that the arithmetic condition
in Zhu's result is unnecessary for the DGS property.
Barranca~\cite{barranca2024} (see  also \cite{BarrancaBarrus2025}) constructed several infinite
families of double stars that are not DS, including
$D(2k,k+1)$ for every $k\ge 2$.
Combining these results with Theorem~\ref{main}, we
conclude that all these double stars are DGS but not DS.
Thus, we obtain several infinite families of graphs
answering the question of Lin et al.~\cite{Lin2026}.

We shall regard double star graphs as a special kind of chain graphs.  A graph is called a \emph{chain graph} if it is $\{2K_2,C_3,C_5\}$-free, i.e., it does not contain any of $2K_2, C_3$ and $C_5$ as an induced subgraph.  A key property of the class of chain graphs is that it is closed under generalized cospectrality \cite{Wang2026}; that is, every graph generalized cospectral with a chain graph is itself a chain graph.  

We note that a chain graph has at most one nontrivial connected
component, since otherwise it would contain $2K_2$ as an induced
subgraph. As we shall show in Proposition~\ref{prop:isolated-vertices},
generalized cospectral chain graphs have the same number of isolated
vertices, and after deleting these isolated vertices, their nontrivial
components remain generalized cospectral. Consequently, the generalized
spectral characterization of chain graphs can be reduced to the
connected case. This contrasts with the ordinary cospectrality. For example, the two chain graphs $K_{1,4}$ and $C_4\cup K_1$ are cospectral, but only the latter graph contains isolated vertices.

As observed in \cite{Bell2008}, the color classes of any nontrivial connected chain graph $G$ can be partitioned into $h$ non-empty cells $U_1, \dots, U_h$ and $V_1, \dots, V_h$ such that $N(u) = V_1 \cup \dots \cup V_{h-i+1}$ for any $u \in U_i$, $1 \le i \le h$. Such a graph is also called a \emph{double nested graph} (DNG). If $s_i = |U_i|$ and $t_i = |V_i|$, we write $G=\mathrm{DNG}(s_1, \dots, s_h; t_1, \dots, t_h)$ (see Fig.~\ref{dng}).
	\begin{figure}[htbp]
	\centering
	\includegraphics[width=0.35\textwidth]{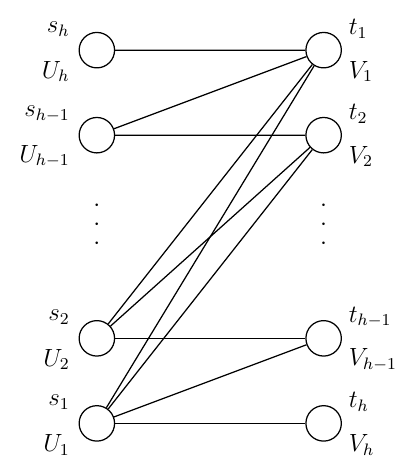}
	\caption{The chain graph $G=\mathrm{DNG}(s_1,\ldots,s_h;t_1,\ldots,t_h)$.}
	\label{dng}
\end{figure}
The case $h=1$ is trivial. Indeed, $\mathrm{DNG}(s,t)$ is the complete
bipartite graph, which is known to be DGS. We mainly consider the case
$h=2$, or equivalently, chain graphs of rank $4$ (see Lemma \ref{lem:rank-chain}). Our main structural result is the following.
\begin{theorem}\label{atmost1}
	Let $G=\mathrm{DNG}(s_1,s_2;t_1,t_2)$. Then $G$ has at most one generalized cospectral mate.
\end{theorem}
We shall give an explicit way to recover the only possible generalized
cospectral mate from four generalized spectral invariants. This also yields
a simple sufficient condition for generalized spectral determination.
\begin{theorem}\label{s1t1}
	If $s_1=t_1\in\{1,2\}$, then
	$\mathrm{DNG}(s_1,s_2;t_1,t_2)$ is DGS.
\end{theorem}
Since $\mathrm{DNG}(1,s;1,t)$ is the double star $D(s,t)$,
Theorem~\ref{main} follows immediately from Theorem~\ref{s1t1}.

The rank-$4$ framework also produces examples beyond double stars.
An\dj eli\'c and Tura \cite{Andelic2025} showed that chain graphs are not, in
general, determined by their adjacency spectra and constructed infinite
families of cospectral chain graphs. A specialization of their construction,
combined with our generalized-cospectrality criterion, gives the following
infinite family of non-double-star graphs.
\begin{theorem}\label{family}
	For every integer $k\ge2$, the chain graph
	\[
	G_k:=\mathrm{DNG}(k,2;2k,k+1)
	\]
	is DGS but not DS.
\end{theorem}
Thus the chain-graph approach not only settles the generalized spectral
problem for double stars, but also provides a new infinite family of DGS but
non-DS graphs within the same rank-$4$ class.

The rest of the paper is organized as follows. In Section~2, we recall
basic facts on equitable partitions, main eigenvalues, and walk generating
functions, and prove that isolated vertices can be cancelled under
generalized cospectrality for chain graphs. In Section~3, we characterize
generalized cospectrality of connected rank-$4$ chain graphs by four explicit
parameter equations. Section~4 reduces this system to a single quadratic
equation and proves Theorem~\ref{atmost1}. In Section~5, we derive the
sufficient condition in Theorem~\ref{s1t1} and prove Theorem~\ref{main} on
double stars. In Section~6, we apply the same framework to the family $G_k$
and prove Theorem~\ref{family}.

\section{Preliminaries}

In this section, we recall some basic facts about equitable partitions,
main eigenvalues, and walk generating functions. We then show that the
generalized spectral characterization of chain graphs can be reduced to
the connected case.

Let $\Pi=\{C_1,\ldots,C_p\}$ be a partition of the vertex set of a graph
$G$. The partition $\Pi$ is called \emph{equitable} if, for every
$i,j\in\{1,\ldots,p\}$, each vertex in $C_i$ has the same number of
neighbors in $C_j$. The corresponding \emph{divisor matrix}
$D=(d_{ij})$ is defined by letting $d_{ij}$ be the number of neighbors
in $C_j$ of any vertex in $C_i$. We refer to
\cite{Cvetkovic2010} for basic results on equitable partitions,
divisors, and main eigenvalues.

Let
\[
G=\mathrm{DNG}(s_1,\ldots,s_h;t_1,\ldots,t_h)
\]
be a connected chain graph of order
\[
n=\sum_{i=1}^{h}(s_i+t_i).
\]
By the definition of a chain graph, the partition
\[
\Pi=\{U_1,\ldots,U_h,V_1,\ldots,V_h\}
\]
is equitable. Let $D$ be the corresponding divisor matrix. It is known
\cite{Andelic2018,Andelic2025} that
\begin{equation}\label{eq:char_divisor}
	\chi(x,G)=x^{\,n-2h}\det(xI_{2h}-D).
\end{equation}
Moreover, a direct calculation gives
\[
\det D=(-1)^h\prod_{i=1}^{h}s_it_i\neq 0.
\]
\begin{lemma}\label{lem:rank-chain}
	Let $G=\mathrm{DNG}(s_1,\ldots,s_h;t_1,\ldots,t_h)$ be a connected
	chain graph. Then
	\[
	\operatorname{rank} A(G)=2h.
	\]
	Consequently, $G$ has rank $4$ if and only if $h=2$.
\end{lemma}

\begin{proof}
	By \eqref{eq:char_divisor},
	\[
	\chi(x,G)=x^{n-2h}\det(xI_{2h}-D).
	\]
	Since
	\[
	\det D=(-1)^h\prod_{i=1}^{h}s_it_i\neq0,
	\]
	the nullity of $A(G)$ is $n-2h$. Hence $\operatorname{rank} A(G)=2h.$
\end{proof}
Recall that an eigenvalue $\lambda$ of a graph $G$ is called
\emph{main} if its eigenspace is not orthogonal to the all-one vector
$\mathbf{1}$. A standard property of equitable partitions states that
every main eigenvalue of $G$ is an eigenvalue of its divisor matrix
$D$; see, for example, \cite[Section~3.9]{Cvetkovic2010}. Since the divisor matrix of a connected chain graph is nonsingular, we obtain the following characterization
of isolated vertices in terms of main eigenvalues.
\begin{lemma}\label{lem:zero-main}
	Let $G$ be a chain graph. Then $0$ is a main
	eigenvalue of $G$ if and only if $G$ has an
	isolated vertex. Consequently, a chain graph
	with isolated vertices cannot be generalized
	cospectral with a chain graph having none.
\end{lemma}

\begin{proof}
	Suppose first that $G$ has no isolated vertices.
	Since a chain graph has at most one nontrivial
	connected component, $G$ is connected.
	Let $D$ be the divisor matrix associated with
	its canonical equitable partition.
	As every main eigenvalue of $G$ is an eigenvalue
	of $D$ and $\det D\neq 0$, it follows that
	$0$ cannot be a main eigenvalue of $G$.
	
	Conversely, suppose that $G$ has an isolated
	vertex $v$, and let $e_v$ be the corresponding
	standard basis vector. Then
	\[
	A(G)e_v=0,
	\qquad
	\mathbf{1}^{\top}e_v=1.
	\]
	Thus $0$ is a main eigenvalue of $G$.
	
	The final assertion follows immediately, since
	generalized cospectral graphs have the same
	main eigenvalues.
\end{proof}
We next recall a useful characterization of generalized cospectrality.
For a graph $G$ with adjacency matrix $A$, let
\[
N_k(G)=\mathbf{1}^{\top}A^k\mathbf{1}
\]
be the total number of walks of length $k$ in $G$, and define its
\emph{walk generating function} by
\begin{equation*}
	W_G(z)
	=\sum_{k=0}^{\infty}N_k(G)z^k
	=\mathbf{1}^{\top}(I-zA)^{-1}\mathbf{1}.
\end{equation*}
It is well known that two cospectral graphs $G$ and $H$ are generalized
cospectral if and only if
\begin{equation*}
	W_G(z)=W_H(z);
\end{equation*}
see \cite{Cvetkovic2010}. Equivalently, generalized cospectral graphs
have the same eigenvalues together with the same main angles, and in
particular they have the same main eigenvalues.

The walk generating function behaves particularly well with respect to
disjoint unions. Namely,
\begin{equation}\label{eq:walk-union}
	W_{G_1\cup G_2}(z)
	=
	W_{G_1}(z)+W_{G_2}(z).
\end{equation}
This gives the following useful cancellation property.

\begin{lemma}\label{lem:component-cancellation}
	Let $G$ and $H$ be generalized cospectral graphs. Suppose that $G_1$
	and $H_1$ are connected components of $G$ and $H$, respectively, such
	that
	\[
	G_1\cong H_1.
	\]
	Then
	\[
	G-G_1
	\quad\text{and}\quad
	H-H_1
	\]
	are generalized cospectral.
\end{lemma}

\begin{proof}
	Write
	\[
	G=G_1\cup G_2,
	\qquad
	H=H_1\cup H_2.
	\]
	Since $G$ and $H$ are cospectral and $G_1\cong H_1$, we have
	\[
	\chi(x,G_1)\chi(x,G_2)
	=
	\chi(x,H_1)\chi(x,H_2),
	\]
	and hence
	\[
	\chi(x,G_2)=\chi(x,H_2).
	\]
	Thus $G_2$ and $H_2$ are cospectral.
	
	Since $G$ and $H$ are generalized cospectral,
	\[
	W_G(z)=W_H(z).
	\]
	By \eqref{eq:walk-union} and $G_1\cong H_1$, we obtain
	\[
	\begin{aligned}
		W_{G_2}(z)
		&=W_G(z)-W_{G_1}(z)\\
		&=W_H(z)-W_{H_1}(z)\\
		&=W_{H_2}(z).
	\end{aligned}
	\]
	Therefore $G_2$ and $H_2$ are cospectral and have the same walk
	generating function. Hence they are generalized cospectral.
\end{proof}

We shall also use the recent result that the class of chain graphs is
closed under generalized cospectrality.

\begin{lemma}[\cite{Wang2026}]\label{lem:chain-closed}
	If a graph $G$ is a chain graph and $H$ is generalized cospectral with
	$G$, then $H$ is also a chain graph.
\end{lemma}

We are now ready to show that isolated vertices play no essential role
in the generalized spectral characterization of chain graphs.

\begin{proposition}\label{prop:isolated-vertices}
	Let $G$ and $H$ be generalized cospectral chain graphs, each having at
	least one edge. Then $G$ and $H$ have the same number of isolated
	vertices. More precisely, if
	\[
	G=G_0\cup kK_1
	\]
	where $G_0$ is a connected chain graph, then
	\[
	H=H_0\cup kK_1
	\]
	for some connected chain graph $H_0$, and $G_0$ and $H_0$ are
	generalized cospectral.
\end{proposition}

\begin{proof}
	Since a chain graph has at most one nontrivial connected component, we
	may write
	\[
	G=G_0\cup kK_1,
	\qquad
	H=H_0\cup \ell K_1,
	\]
	where $G_0$ and $H_0$ are connected chain graphs.
	
	Suppose, to the contrary, that $k\ne\ell$. By symmetry, we may assume
	that $k>\ell$. Applying Lemma~\ref{lem:component-cancellation}
	successively to $\ell$ isolated vertices, we conclude that
	\[
	G_0\cup (k-\ell)K_1
	\quad\text{and}\quad
	H_0
	\]
	are generalized cospectral.	Since $k-\ell>0$, the graph $G_0\cup (k-\ell)K_1$
	has an isolated vertex, whereas $H_0$ has none.
	By Lemma~\ref{lem:zero-main}, these two graphs
	cannot be generalized cospectral, a contradiction. Hence $k=\ell.$
	
	Finally, applying Lemma~\ref{lem:component-cancellation} successively
	to the $k$ isolated vertices shows that $G_0$ and $H_0$ are generalized
	cospectral.
\end{proof}

Proposition~\ref{prop:isolated-vertices} shows that, in studying the
generalized spectral characterization of chain graphs, it is sufficient
to restrict our attention to connected chain graphs.

For later use, we record the characteristic polynomial in the case
$h=2$. Let
\[
G=\mathrm{DNG}(s_1,s_2;t_1,t_2).
\]
With the cells ordered as $U_1,U_2,V_1,V_2$, the corresponding divisor
matrix is
\begin{equation*}
	D=
	\begin{pmatrix}
		0   & 0   & t_1 & t_2\\
		0   & 0   & t_1 & 0\\
		s_1 & s_2 & 0   & 0\\
		s_1 & 0   & 0   & 0
	\end{pmatrix}.
\end{equation*}
A direct computation gives
\begin{equation}\label{eq:charpoly_D}
	\det(xI_4-D)
	=
	x^4-(s_1t_1+s_1t_2+s_2t_1)x^2+s_1s_2t_1t_2.
\end{equation}
Combining this with \eqref{eq:char_divisor}, we obtain
\begin{equation}\label{f2}
	\chi(x,G)
	=
	x^{n-4}
	\bigl(
	x^4-(s_1t_1+s_1t_2+s_2t_1)x^2
	+s_1s_2t_1t_2
	\bigr).
\end{equation}
\section{Generalized cospectral mates of rank-$4$ chain graphs}
\label{sec:rank4}

In this section, we study generalized cospectral mates of connected
chain graphs of rank $4$. By Lemma~\ref{lem:rank-chain}, every connected
chain graph of rank $4$ is of the form
\[
G=\mathrm{DNG}(s_1,s_2;t_1,t_2),
\]
where $s_1,s_2,t_1,t_2$ are positive integers.

Moreover, by Lemma~\ref{lem:chain-closed} and
Proposition~\ref{prop:isolated-vertices}, every generalized cospectral
mate of $G$ is again a connected chain graph. Since cospectral graphs
have the same rank, any such mate is also of rank $4$, and hence can be
written as
\[
H=\mathrm{DNG}(s_1',s_2';t_1',t_2').
\]

For convenience, put
\begin{equation}\label{eq:nmbc}
	\begin{aligned}
		n&=s_1+s_2+t_1+t_2,\\
		m&=s_1t_1+s_1t_2+s_2t_1,\\
		c&=s_1s_2t_1t_2,\\
		b&=s_2t_2(s_1+t_1).	
	\end{aligned}
\end{equation}
Here $n$ and $m$ are the number of vertices and the number of edges of
$G$, respectively; the role of $b$ will become clear below. By \eqref{f2},
\begin{equation}\label{eq:rank4_charpoly}
	\chi(x,G)
	=
	x^{n-4}\left(x^4-mx^2+c\right).
\end{equation}

The following observation gives a characterization of ordinary
cospectrality within this class.

\begin{lemma}\label{lem:rank4_cospectral}
	Let
	\[
	G=\mathrm{DNG}(s_1,s_2;t_1,t_2)
	\quad\text{and}\quad
	H=\mathrm{DNG}(s_1',s_2';t_1',t_2')
	\]
	be connected chain graphs of rank $4$. Then $G$ and $H$ are cospectral
	if and only if
	\begin{align}
		s_1+s_2+t_1+t_2
		&=
		s_1'+s_2'+t_1'+t_2',
		\label{eq:condition1}\\
		s_1t_1+s_1t_2+s_2t_1
		&=
		s_1't_1'+s_1't_2'+s_2't_1',
		\label{eq:condition2}\\
		s_1s_2t_1t_2
		&=
		s_1's_2't_1't_2'.
		\label{eq:condition3}
	\end{align}
\end{lemma}

\begin{proof}
	This follows immediately from \eqref{eq:rank4_charpoly}, since the
	three quantities in \eqref{eq:condition1}--\eqref{eq:condition3}
	determine the characteristic polynomial of a rank-$4$ connected
	chain graph.
\end{proof}

Thus far only the ordinary spectrum has been used. To obtain
an additional condition from the generalized spectrum, we use
the walk numbers defined in Section~2. Recall that generalized
cospectral graphs have the same walk generating function and
hence the same number of walks of each length.

\begin{lemma}\label{lem:walk3_rank4}
	Let $G=\mathrm{DNG}(s_1,s_2;t_1,t_2)$, and let
	$n,m,b,c$ be as in \eqref{eq:nmbc}. Then
	\begin{align}
		N_0(G)&=n,\label{eq:N0}\\
		N_1(G)&=2m,\label{eq:N1}\\
		N_2(G)&=nm-b,\label{eq:N2}\\
		N_3(G)&=2(m^2-c).\label{eq:N3}
	\end{align}
\end{lemma}

\begin{proof}
	The first two identities are immediate. Put
	\[
	S=s_1+s_2,\qquad T=t_1+t_2.
	\]
	The vertices in $U_1,U_2,V_1,V_2$ have degrees
	$T,t_1,S,s_1$, respectively. Since
	\[
	N_2(G)
	=\mathbf{1}^{\top}A(G)^2\mathbf{1}
	=\sum_{v\in V(G)}d_G(v)^2
	\]
	and
	\[
	m=s_1T+s_2t_1=t_1S+s_1t_2,
	\]
	we obtain
	\begin{align*}
		N_2(G)
		&=s_1T^2+s_2t_1^2+t_1S^2+t_2s_1^2\\
		&=m(S+T)-s_2t_2(s_1+t_1)\\
		&=nm-b.
	\end{align*}
	
	Finally, counting walks of length $3$ by their middle edge
	gives
	\begin{align*}
		\frac{1}{2}N_3(G)
		&=\sum_{uv\in E(G)}d_G(u)d_G(v)\\
		&=s_1t_1TS+s_1^2t_2T+s_2t_1^2S\\
		&=(s_1T+s_2t_1)(t_1S+s_1t_2)-s_1s_2t_1t_2\\
		&=m^2-c.
	\end{align*}
	This proves \eqref{eq:N3}.
\end{proof}

Consequently, if $G$ and $H$ are generalized cospectral, then
\eqref{eq:condition1}--\eqref{eq:condition3} hold and
$N_2(G)=N_2(H)$. Since $n$ and $m$ are the same for both
graphs, \eqref{eq:N2} yields
\begin{equation}\label{eq:condition4}
	s_2t_2(s_1+t_1)
	=
	s_2't_2'(s_1'+t_1').
\end{equation}
We now show that these four conditions are also sufficient
by deriving an explicit formula for the walk generating
function.
\begin{lemma}\label{lem:walk_gen_rank4}
	Let
	\[
	G=\mathrm{DNG}(s_1,s_2;t_1,t_2)
	\]
	be a connected chain graph of rank $4$, and let $n,m,b,c$ be as in
	\eqref{eq:nmbc}. Then
	\begin{equation}\label{eq:walk_gen_formula}
		W_G(z)
		=
		\frac{
			n+2mz-b z^2-2cz^3
		}{
			1-mz^2+cz^4
		}.
	\end{equation}
\end{lemma}

\begin{proof}
Let $D$ be the divisor matrix of $G$ corresponding to the
equitable partition
\[
\{U_1,U_2,V_1,V_2\},
\]
and let $C$ be its characteristic matrix.
Since the partition is equitable, we have
\[
AC=CD,
\]
where $A=A(G)$. Consequently, for any polynomial
$f(x)$, we have
\[
f(A)C=Cf(D).
\]
By \eqref{eq:charpoly_D}, the characteristic polynomial of $D$ is
\[
\det(xI_4-D)=x^4-mx^2+c,
\]
with $m$ and $c$ as in \eqref{eq:nmbc}. Hence, by the
Cayley--Hamilton theorem,
\[
D^4-mD^2+cI_4=0,
\]
which implies
\[
(A^4-mA^2+cI)C=0.
\]
Since $\mathbf{1}=C\mathbf{1}_4$, it follows that
\begin{equation}\label{eq:annihilate_one}
	A^4\mathbf{1}
	=
	mA^2\mathbf{1}-c\mathbf{1}.
\end{equation}
	Multiplying \eqref{eq:annihilate_one} by $A^k$ and then by
	$\mathbf{1}^\top$, we obtain
	\begin{equation}\label{eq:walk_recurrence}
		N_{k+4}(G)
		=
		mN_{k+2}(G)-cN_k(G),
		\qquad k\geq0.
	\end{equation}
	Multiplying the series
	\[
	W_G(z)=\sum_{k\geq0}N_k(G)z^k
	\]
	by $1-mz^2+cz^4$ and using \eqref{eq:N0}--\eqref{eq:N3} and
	\eqref{eq:walk_recurrence}, we obtain
	\begin{align*}
		(1-mz^2+cz^4)W_G(z)&=\sum_{k\ge 0} N_kz^k-m\sum_{k\ge 0} N_kz^{k+2}+c\sum_{k\ge 0} N_kz^{k+4}\\
		&=\sum_{k=0}^3 N_kz^k-m\sum_{k= 0}^1 N_kz^{k+2}\\
		&=N_0+N_1z+(N_2-mN_0)z^2+(N_3-mN_1)z^3\\
		&=n+2mz-bz^2-2cz^3.
	\end{align*}
	This proves \eqref{eq:walk_gen_formula}.
\end{proof}

We can now give the desired characterization.

\begin{theorem}\label{thm:rank4_gcospectral}
	Let
	\[
	G=\mathrm{DNG}(s_1,s_2;t_1,t_2)
	\]
	and
	\[
	H=\mathrm{DNG}(s_1',s_2';t_1',t_2')
	\]
	be connected chain graphs of rank $4$. Then $G$ and $H$ are
	generalized cospectral if and only if the following four equalities
	hold:
	\begin{align}
		s_1+s_2+t_1+t_2
		&=
		s_1'+s_2'+t_1'+t_2',
		\label{eq:main_condition1}\\
		s_1t_1+s_1t_2+s_2t_1
		&=
		s_1't_1'+s_1't_2'+s_2't_1',
		\label{eq:main_condition2}\\
		s_1s_2t_1t_2
		&=
		s_1's_2't_1't_2',
		\label{eq:main_condition3}\\
		s_2t_2(s_1+t_1)
		&=
		s_2't_2'(s_1'+t_1').
		\label{eq:main_condition4}
	\end{align}
\end{theorem}

\begin{proof}
Suppose first that $G$ and $H$ are generalized cospectral. Then
\eqref{eq:main_condition1}--\eqref{eq:main_condition3} follow
directly from Lemma~\ref{lem:rank4_cospectral}, while
\eqref{eq:main_condition4} is exactly \eqref{eq:condition4}.
	
	Conversely, suppose that
	\eqref{eq:main_condition1}--\eqref{eq:main_condition4} hold.
	By the first three equalities and
	Lemma~\ref{lem:rank4_cospectral}, $G$ and $H$ are cospectral.
	
	Let $n,m,c,b$ denote the four quantities appearing in \eqref{eq:nmbc}. By
	Lemma~\ref{lem:walk_gen_rank4},
	\[
	W_G(z)
	=
	\frac{
		n+2mz-b z^2-2cz^3
	}{
		1-mz^2+cz^4
	}
	=
	W_H(z).
	\]
	Thus $G$ and $H$ are cospectral and have the same walk generating
	function. It follows that they are generalized cospectral.
\end{proof}
\section{The number of generalized cospectral mates}
\label{sec:number_mates}

In this section, we show that a connected chain graph of rank $4$
has at most one non-isomorphic generalized cospectral mate. The proof
is constructive and also provides an efficient procedure for finding
the possible mate.

Let $G=\mathrm{DNG}(s_1,s_2;t_1,t_2)$ be fixed, and let
$n,m,c,b$ be as in \eqref{eq:nmbc}.
By Theorem~\ref{thm:rank4_gcospectral}, a connected rank-$4$
chain graph
\[
H=\mathrm{DNG}(x_1,x_2;y_1,y_2)
\]
is generalized cospectral with $G$ if and only if
\begin{align}
	x_1+x_2+y_1+y_2&=n, \label{eq:sys_n}\\
	x_1y_1+x_1y_2+x_2y_1&=m, \label{eq:sys_m}\\
	x_1x_2y_1y_2&=c, \label{eq:sys_c}\\
	x_2y_2(x_1+y_1)&=b. \label{eq:sys_b}
\end{align}

Since interchanging the two parts of a bipartite graph does not change
the graph up to isomorphism, we identify
\[
(x_1,x_2;y_1,y_2)
\quad\text{and}\quad
(y_1,y_2;x_1,x_2).
\]
Thus, whenever convenient, we may assume that
\[
x_1\leq y_1.
\]

We first reduce the four-variable system
\eqref{eq:sys_n}--\eqref{eq:sys_b} to a quadratic equation in the
single variable $x_2y_2$.

\begin{lemma}\label{lem:quadratic_z}
	Let $(x_1,x_2;y_1,y_2)\in\mathbb{R}_{>0}^4$ satisfy
	\eqref{eq:sys_n}--\eqref{eq:sys_b}, and put
	\[
	\xi=x_2y_2.
	\]
	Then $\xi$ is a root of
	\begin{equation}\label{eq:Phi}
		\Phi(\xi)
		=
		(m^2-4c)\xi^2
		+
		(b^2+cn^2-mnb-2mc)\xi
		+
		(c^2-cnb+mb^2)
		=0.
	\end{equation}
	Moreover,
	\[
	m^2-4c>0,
	\]
	so $\Phi$ is always a quadratic polynomial.
\end{lemma}

\begin{proof}
	Set
	\[
	u=x_1+y_1,\qquad
	v=x_1y_1,\qquad
	w=x_2+y_2.
	\]
	Since $\xi=x_2y_2$, equations \eqref{eq:sys_c} and
	\eqref{eq:sys_b} give
	\begin{equation}\label{eq:uv}
		u=\frac{b}{\xi},
		\qquad
		v=\frac{c}{\xi},
	\end{equation}
	while \eqref{eq:sys_n} yields
	\begin{equation}\label{eq:w_sum}
		w=n-\frac{b}{\xi}.
	\end{equation}
	Moreover, by \eqref{eq:sys_m},
	\[
	x_1y_2+x_2y_1=m-v.
	\]
	Hence
	\[
	(x_1-y_1)(x_2-y_2)
	=(x_1+y_1)(x_2+y_2)
	-2(x_1y_2+x_2y_1)
	=uw-2(m-v).
	\]
	Squaring both sides and using
	\[
	(x_1-y_1)^2=u^2-4v,
	\qquad
	(x_2-y_2)^2=w^2-4\xi,
	\]
	we obtain
	\begin{equation}\label{eq:disc_identity}
		(u^2-4v)(w^2-4\xi)
		=
		\bigl(uw-2(m-v)\bigr)^2.
	\end{equation}
	
	We now substitute \eqref{eq:uv} and \eqref{eq:w_sum} into
	\eqref{eq:disc_identity}. The three factors appearing above become
	\[
	u^2-4v
	=
	\frac{b^2-4c\xi}{\xi^2},
	\]
	\[
	w^2-4\xi
	=
	\frac{n^2\xi^2-2nb\xi+b^2-4\xi^3}{\xi^2},
	\]
	and
	\[
	uw-2(m-v)
	=
	\frac{-2m\xi^2+(nb+2c)\xi-b^2}{\xi^2}.
	\]
	Therefore, after clearing denominators, we obtain
	\[
	(b^2-4c\xi)
	\bigl(n^2\xi^2-2nb\xi+b^2-4\xi^3\bigr)
	=
	\bigl(-2m\xi^2+(nb+2c)\xi-b^2\bigr)^2.
	\]
	Expanding and simplifying, the difference between
	the two sides becomes
	\[
	-4\xi^2\Bigl[
	(m^2-4c)\xi^2
	+(b^2+cn^2-mnb-2mc)\xi
	+(c^2-cnb+mb^2)
	\Bigr].
	\]
	Since $\xi>0$, we obtain \eqref{eq:Phi}.
	
	Finally, since
	\[
	m=s_1t_1+s_1t_2+s_2t_1
	>s_1t_2+s_2t_1
	\ge 2\sqrt{s_1s_2t_1t_2}
	=2\sqrt c,
	\]
	we have $m^2-4c>0$.
\end{proof}
Once a root $\xi$ of \eqref{eq:Phi} is known, the four parameters can
be recovered almost completely.
\begin{lemma}\label{lem:reconstruct}
	Let $\xi>0$ correspond to a positive real solution of
	\eqref{eq:sys_n}--\eqref{eq:sys_b}. Then
	$x_1,y_1$ are the roots of
	\begin{equation}\label{eq:quad_xy1}
		Z^2-\frac{b}{\xi}Z+\frac{c}{\xi}=0,
	\end{equation}
	whereas $x_2,y_2$ are the roots of
	\begin{equation}\label{eq:quad_xy2}
		Z^2-
		\left(n-\frac{b}{\xi}\right)Z+\xi=0.
	\end{equation}
	For a fixed $\xi$, these two unordered pairs determine
	at most one solution of
	\eqref{eq:sys_n}--\eqref{eq:sys_b}
	up to interchange of the two bipartition classes.
\end{lemma}
\begin{proof}
	The first assertion follows immediately from
	\[
	x_1+y_1=\frac{b}{\xi},
	\qquad
	x_1y_1=\frac{c}{\xi},
	\]
	and the second follows from
	\[
	x_2+y_2=n-\frac{b}{\xi},
	\qquad
	x_2y_2=\xi.
	\]
	It remains only to determine how the two roots of
	\eqref{eq:quad_xy2} are assigned to $x_2$ and $y_2$.	We may assume that $x_1\leq y_1$. If $x_1=y_1$, then interchanging
	$x_2$ and $y_2$ amounts precisely to interchanging the two
	bipartition classes, and hence gives an equivalent solution.
	
	Suppose now that $x_1<y_1$. If $x_2\neq y_2$, the two possible
	assignments of the roots of \eqref{eq:quad_xy2} give two values for
	the left-hand side of \eqref{eq:sys_m}. Their difference is
	\[
	(x_1-y_1)(y_2-x_2),
	\]
	which is nonzero. Therefore, at most one of the two assignments can
	satisfy \eqref{eq:sys_m}. If $x_2=y_2$, there is of course only one
	assignment.	Hence a fixed value of $\xi$ gives at most one solution class.
\end{proof}

We can now prove the main result of this section.

\begin{lemma}\label{lem:atmost_two_solutions}
	For fixed positive parameters $s_1,s_2,t_1,t_2$, the system
	\eqref{eq:sys_n}--\eqref{eq:sys_b} has at most two equivalence
	classes of positive real solutions.
\end{lemma}

\begin{proof}
	By Lemma~\ref{lem:quadratic_z}, the quantity
	\[
	\xi=x_2y_2
	\]
	associated with any positive real solution must be a root of the
	quadratic polynomial $\Phi$. Since $\Phi$ has at most two real roots,
	and since Lemma~\ref{lem:reconstruct} shows that each positive root of
	$\Phi$ gives at most one solution class, the system has at most two
	equivalence classes of positive real solutions.
\end{proof}

\begin{proof}[Proof of Theorem~\ref{atmost1}]
	The original quadruple
	\[
	(s_1,s_2;t_1,t_2)
	\]
	is one solution of \eqref{eq:sys_n}--\eqref{eq:sys_b}. By
	Lemma~\ref{lem:atmost_two_solutions}, there can be at most one
	additional solution class. By Theorem~\ref{thm:rank4_gcospectral},
	this means that $G$ has at most one non-isomorphic generalized
	cospectral mate.
\end{proof}

The proof above is constructive. In fact, the original value
\[
\xi_0=s_2t_2
\]
is already one root of $\Phi$. Hence, by Vieta's formula, the only
possible second root is
\begin{equation}\label{eq:second_root}
	\xi_1=
	\frac{c^2-cnb+mb^2}
	{(m^2-4c)s_2t_2}.
\end{equation}
Thus the possible generalized cospectral mate, if it exists, can be
reconstructed by solving two quadratic equations.

More precisely, one first computes $\xi_1$ from
\eqref{eq:second_root}. If $\xi_1$ is not a positive integer
or $\xi_1=\xi_0$, then $G$ is DGS. Otherwise, one solves
\[
Z^2-\frac{b}{\xi_1}Z+\frac{c}{\xi_1}=0
\]
for $x_1,y_1$, and
\[
Z^2-
\left(n-\frac{b}{\xi_1}\right)Z+\xi_1=0
\]
for $x_2,y_2$.
A generalized cospectral mate exists if and only if
the roots of both equations are positive integers and can be
assigned to $x_1,y_1$ and $x_2,y_2$, respectively, so that
\eqref{eq:sys_m} holds. In this case,
$\mathrm{DNG}(x_1,x_2;y_1,y_2)$ is the unique generalized
cospectral mate of $G$.
\section{A sufficient condition for being DGS}

In this section we give a  simple condition which guarantees
uniqueness. It will also imply the result on double star graphs stated
in the Introduction.

\begin{proof}[Proof of Theorem~\ref{s1t1}]
	Set
	\[
	s_1=t_1=k,
	\qquad k\in\{1,2\}.
	\]
	Let $H$ be any graph generalized cospectral with $G$.
	By Lemma~\ref{lem:chain-closed} and Proposition~\ref{prop:isolated-vertices},
	$H$ is a connected chain graph. Since $H$ is cospectral with $G$,
	it also has rank $4$. Hence
	\[
	H=\mathrm{DNG}(x_1,x_2;y_1,y_2)
	\]
	for some positive integers $x_1,x_2,y_1,y_2$. By
	Theorem~\ref{thm:rank4_gcospectral}, the parameters of $H$ satisfy
	\eqref{eq:sys_n}--\eqref{eq:sys_b}.
	
	For $G$, we have
	\[
	c=k^2s_2t_2,
	\qquad
	b=2ks_2t_2.
	\]
	Hence
	\begin{equation}\label{eq:b_c_ratio}
		\frac{b}{c}=\frac{2}{k}.
	\end{equation}
	For $H$, equations \eqref{eq:sys_c} and
	\eqref{eq:sys_b} give
	\[
	\frac{b}{c}
	=
	\frac{x_1+y_1}{x_1y_1}.
	\]
	Combining this with \eqref{eq:b_c_ratio}, we obtain
$
	k(x_1+y_1)=2x_1y_1
$,
	or equivalently,
	\begin{equation}\label{eq:kfactor}
		(2x_1-k)(2y_1-k)=k^2.
	\end{equation}
	
	If $k=1$, then
	\[
	(2x_1-1)(2y_1-1)=1.
	\]
	Since $x_1$ and $y_1$ are positive integers, we must have $
	x_1=y_1=1$.
	
	If $k=2$, then \eqref{eq:kfactor} becomes $
	(x_1-1)(y_1-1)=1$. 
	Again, positivity and integrality imply
	$	x_1=y_1=2$.
	
	Thus, in either case, $
	x_1=y_1=k=s_1=t_1$.	Now \eqref{eq:sys_c} gives $
	x_2y_2=s_2t_2,
	$ while \eqref{eq:sys_n} gives
	$
	x_2+y_2=s_2+t_2
	$. 
	Therefore
	\[
	\{x_2,y_2\}=\{s_2,t_2\}.
	\]
	Consequently, either
	\[
	(x_1,x_2;y_1,y_2)
	=
	(s_1,s_2;t_1,t_2),
	\]
	or it is obtained from this quadruple by interchanging the two
	bipartition classes. In either case,
	$
	H\cong G
	$.	Hence $G$ is DGS.
\end{proof}
\begin{proof}[Proof of Theorem~\ref{main}]
	Every double star $D(s,t)$ can be written as
$	\mathrm{DNG}(1,s;1,t)$. 
	Thus Theorem~\ref{s1t1}, with $s_1=t_1=1$, shows that every double
	star is DGS.
\end{proof}
The following example shows that Theorem~\ref{s1t1} does not
extend to the case $s_1=t_1=3$ and that the bound in
Theorem~\ref{atmost1} is sharp.
\begin{example}\label{ex:k3}\normalfont
	Consider
	\[
	G=\mathrm{DNG}(3,8;3,3).
	\]
	For this graph, the four generalized spectral invariants introduced
	in Section~\ref{sec:rank4} are
	\[
	n=17,\qquad m=42,\qquad c=216,\qquad b=144.
	\]
	The quadratic polynomial in Lemma~\ref{lem:quadratic_z} becomes
	\[
	\Phi(\xi)=900(\xi-24)(\xi-18).
	\]
	The root $\xi=24$ corresponds to the original graph, since
	$s_2t_2=8\cdot3=24$. Consider the other root $\xi=18$.
	By Lemma~\ref{lem:reconstruct},
	\[
	x_1+y_1=\frac{b}{\xi}=8,
	\qquad
	x_1y_1=\frac{c}{\xi}=12,
	\]
	and hence
	\[
	\{x_1,y_1\}=\{2,6\}.
	\]
	Similarly,
	\[
	x_2+y_2=n-\frac{b}{\xi}=9,
	\qquad
	x_2y_2=\xi=18,
	\]
	so that
	\[
	\{x_2,y_2\}=\{3,6\}.
	\]
	Choosing
	\[
	(x_1,x_2;y_1,y_2)=(2,3;6,6),
	\]
	we have
	\[
	x_1y_1+x_1y_2+x_2y_1
	=2\cdot6+2\cdot6+3\cdot6
	=42=m.
	\]
	Therefore
	\[
	H=\mathrm{DNG}(2,3;6,6)
	\]
	satisfies all four equations in
	Theorem~\ref{thm:rank4_gcospectral}. Hence $G$ and $H$ are
	generalized cospectral.
	
	The graphs $G$ and $H$ are not isomorphic,
	since their bipartition sizes are
	$\{11,6\}$ and $\{5,12\}$, respectively.
	Thus $G$ is not DGS.
\end{example}
\section{Further DGS but non-DS chain graphs}
\label{sec:family}
We finish by combining the quadratic reduction of
Section~\ref{sec:number_mates} with a known cospectral construction.
This gives a sufficient condition for further rank-$4$ chain graphs
to be DGS but not DS, with the family in Theorem~\ref{family} as a
special case.

The following cospectral construction is due to
An\dj eli\'c and Tura \cite[Theorem~3.5]{Andelic2025}; we include the
short proof for completeness.
\begin{lemma}[\cite{Andelic2025}]\label{lem:AT_construction}
	Let $\ell,t_1,t_2,s_1$ be positive integers with $\ell\ge2$ and
	\begin{equation}\label{eq:AT_conditions}
		\ell(s_1-t_1)-t_2(\ell-1)=0.
	\end{equation}
	Then the connected chain graphs
	\begin{equation}\label{eq:AT_pair}
		G=\mathrm{DNG}(t_1,t_2;\ell t_1,s_1),
		\qquad
		H=\mathrm{DNG}(s_1,t_2;t_1,\ell t_1)
	\end{equation}
	are cospectral.
\end{lemma}

\begin{proof}
	The graphs $G$ and $H$ have the same values of $n$ and $c$.
	Moreover, by \eqref{eq:AT_conditions},
	\begin{align*}
		m(G)-m(H)
		&=\ell t_1^2+\ell t_1t_2-\ell s_1t_1-t_1t_2\\
		&=-t_1\bigl(\ell(s_1-t_1)-t_2(\ell-1)\bigr)=0.
	\end{align*}
	Hence they are cospectral by
	Lemma~\ref{lem:rank4_cospectral}.
\end{proof}

The graphs in Lemma~\ref{lem:AT_construction} need not be
non-isomorphic. Indeed, if $s_1=\ell t_1$, then
\eqref{eq:AT_conditions} gives $t_2=\ell t_1=s_1$, and interchanging
the two color classes gives an isomorphism between $G$ and $H$.
The next result excludes this case and uses only the second-root
test from Section~\ref{sec:number_mates}.

For the graph $G$ in \eqref{eq:AT_pair}, the invariants defined in
\eqref{eq:nmbc} are
\begin{equation}\label{eq:AT_invariants}
	\begin{aligned}
		n&=(\ell+1)t_1+s_1+t_2, &
		m&=t_1(\ell t_1+s_1+\ell t_2),\\
		c&=\ell s_1t_1^2t_2, &	b&=(\ell+1)s_1t_1t_2.
	\end{aligned}
\end{equation}

\begin{theorem}\label{thm:AT_DGS}
	Let $G$ be as in Lemma~\ref{lem:AT_construction}, with
	$s_1\ne\ell t_1$. If
	\begin{equation}\label{eq:AT_DGS_test}
		\xi_1=
		\frac{c^2-cnb+mb^2}
		{(m^2-4c)s_1t_2}
		\notin\mathbb{Z}_{>0},
	\end{equation}
	where $n,m,c,b$ are given by \eqref{eq:AT_invariants}, then
	$G$ is DGS but not DS. Moreover, $G$ is not a double star,
	and both $G$ and its complement are connected.
\end{theorem}

\begin{proof}
	Let $H$ be the graph in \eqref{eq:AT_pair}. By
	Lemma~\ref{lem:AT_construction}, $G$ and $H$ are cospectral.
	However,
	\begin{align*}
		b(H)-b(G)
		&=\ell t_1t_2(s_1+t_1)-(\ell+1)s_1t_1t_2\\
		&=t_1t_2(\ell t_1-s_1)\ne0.
	\end{align*}
	Thus Theorem~\ref{thm:rank4_gcospectral} shows that they are
	not generalized cospectral, and in particular they are not
	isomorphic. Hence $G$ is not DS.
	
	Now let $K$ be any graph generalized cospectral with $G$.
	By Lemma~\ref{lem:chain-closed} and
	Proposition~\ref{prop:isolated-vertices}, $K$ is a connected
	chain graph. Since its adjacency rank is $4$, we may write
	\[
	K=\mathrm{DNG}(x_1,x_2;y_1,y_2).
	\]
	By Lemma~\ref{lem:quadratic_z} and \eqref{eq:second_root},
	\[
	x_2y_2\in\{\xi_0,\xi_1\},
	\qquad \xi_0=s_1t_2.
	\]
	Since $x_2y_2$ is a positive integer, condition
	\eqref{eq:AT_DGS_test} forces $x_2y_2=\xi_0$.
	Lemma~\ref{lem:reconstruct} then implies that the parameter
	quadruple of $K$ agrees with that of $G$ up to interchange of
	the two color classes. Thus $K\cong G$, proving that $G$ is DGS.
	
	Finally, $\ell t_1\ge2$, and $V_1$ is complete to
	$U_1\cup U_2$, which contains at least two vertices. Hence
	$G$ contains a $4$-cycle and is not a double star.
	The graph $G$ is connected by construction. In its complement,
	each color class induces a clique, and all edges between
	$U_2$ and $V_2$ are present. Thus its complement is also connected.
\end{proof}
For example, take $\ell=3$, $t_1=2$, $t_2=3$, and $s_1=4$. Then
$G=\mathrm{DNG}(2,3;6,4)$ is one of the cospectral graphs in
\cite[Example~3.4]{Andelic2025}. Here
\[
n=15,\qquad m=38,\qquad b=96,\qquad c=144,
\qquad \xi_1=\frac{3408}{217}\notin\mathbb{Z}_{>0}.
\]
Thus Theorem~\ref{thm:AT_DGS} shows that $G$ is DGS but not DS.

We now apply the criterion to the family $G_k$ in
Theorem~\ref{family}. The following arithmetic fact is the
remaining step.

\begin{lemma}\label{lem:noninteger_root}
	For every integer $k\ge2$, the number
	\begin{equation}\label{eq:rk}
		r_k=
		\frac{k(14k^2+68k+54)}{9k^2+14k+9}
	\end{equation}
	is not an integer.
\end{lemma}

\begin{proof}
	Put
	\[
	P=14k^2+68k+54,
	\qquad Q=9k^2+14k+9.
	\]
	Suppose that $r_k=kP/Q$ is an integer. Since
	$\gcd(k,Q)=\gcd(k,9)$, putting
	$g=\gcd(k,9)\in\{1,3,9\}$ gives $\frac{Q}{g}\mid P$.
	Thus
	\[
	P=r\frac{Q}{g}
	\]
	for some positive integer $r$.
	
	If $g=1$, then $Q<P<4Q$ for $k\ge2$, so $r\in\{2,3\}$.
	The equations $P=2Q$ and $P=3Q$ reduce, respectively, to
	\[
	k^2-10k-9=0,
	\qquad 13k^2-26k-27=0.
	\]
	Their discriminants are $136$ and $2080$, neither of which is
	a perfect square.
	
	If $g=3$, write $k=3u$, where $3\nmid u$. Then
	\[
	Q/3=27u^2+14u+3,
	\qquad P=126u^2+204u+54.
	\]
	Since $4Q/3<P<9Q/3$, we have $r\in\{5,6,7,8\}$.
	Modulo $3$, however, $P\equiv0$ and
	$Q/3\equiv2u\not\equiv0$, so $3\mid r$.
	Hence $r=6$, giving
	\[
	3u^2-10u-3=0.
	\]
	Its discriminant is $136$, again not a perfect square.
	
	Finally, if $g=9$, write $k=9u$. Then
	\[
	Q/9=81u^2+14u+1,
	\qquad P=1134u^2+612u+54.
	\]
	For $u\ge1$, we have $14Q/9<P<19Q/9$, so
	$r\in\{15,16,17,18\}$. The corresponding equations are
	\[
	\begin{aligned}
		27u^2-134u-13&=0, &
		81u^2-194u-19&=0,\\
		243u^2-374u-37&=0, &
		9u^2-10u-1&=0.
	\end{aligned}
	\]
	Their discriminants are, respectively,
	\[
	19360,\qquad 43792,\qquad 175840,\qquad 136,
	\]
	none of which is a perfect square. All cases are impossible,
	proving the lemma.
\end{proof}

\begin{proof}[Proof of Theorem~\ref{family}]
	Let $k\ge2$. Apply Lemma~\ref{lem:AT_construction} with
	\[
	\ell=2,\qquad t_1=k,\qquad t_2=2,\qquad s_1=k+1.
	\]
	Then \eqref{eq:AT_conditions} holds, and $s_1\ne\ell t_1$.
	The corresponding graphs in \eqref{eq:AT_pair} are
	\begin{equation}\label{eq:GkHk}
		G_k=\mathrm{DNG}(k,2;2k,k+1),
		\qquad
		H_k=\mathrm{DNG}(k+1,2;k,2k).
	\end{equation}
	For $G_k$, we have
	\[
	\begin{aligned}
		n&=4k+3, & m&=3k^2+5k,\\
		b&=6k(k+1), & c&=4k^2(k+1).
	\end{aligned}
	\]
	Since $\xi_0=2(k+1)$, formula~\eqref{eq:second_root} gives
	\[
	\xi_1=
	\frac{2k(k+1)(7k+27)}{9k^2+14k+9}
	=r_k.
	\]
	By Lemma~\ref{lem:noninteger_root}, this is not an integer.
	Theorem~\ref{thm:AT_DGS} now applies and proves the result.
\end{proof}
\section*{Declaration of competing interest}
	The authors declare that they have no known competing financial interests or personal relationships that could have appeared to influence the work reported in this paper.
	
	\section*{Acknowledgments}
	This work is partially supported by the National Natural Science Foundation of China (Grant No.~12001006) and the Wuhu Science and Technology Project, China (Grant No.~2024kj015).
	
	\section*{Declaration of generative AI and AI-assisted technologies in the manuscript preparation process}
	During the development and preparation of this work, the authors used GPT and Gemini as auxiliary tools for exploratory computation, counterexample testing, and language refinement. The research problem, mathematical ideas, and proof strategy were developed by the authors. All AI-generated outputs were independently examined and verified by the authors. The authors finalized the results and proofs and take full responsibility for the originality, accuracy, and integrity of the manuscript.

\end{document}